\documentclass[reqno]{amsart}

\usepackage{amsmath}
\usepackage{amsfonts}
\usepackage{amssymb,enumerate}
\usepackage{amsthm}
\usepackage[all]{xy}
\usepackage{rotating}
\usepackage{hyperref}
\usepackage{color}

\theoremstyle{plain}
\newtheorem{lem}{Lemma}[section]
\newtheorem{cor}[lem]{Corollary}
\newtheorem{prop}[lem]{Proposition}
\newtheorem{thm}[lem]{Theorem}

\newtheorem*{mthm*}{Main Theorem}

\theoremstyle{definition}
\newtheorem{defn}[lem]{Definition}
\newtheorem{ex}[lem]{Example}

\newtheorem{disc}[lem]{Remark}

\newtheorem{para}[lem]{}

\newtheorem*{convention*}{Convention}

\newcommand{\id}{\operatorname{id}}

\newcommand{\HH}{\operatorname{H}}

\newcommand{\im}{\operatorname{Im}}
\newcommand{\shift}{\mathsf{\Sigma}}

\newcommand{\Ker}{\operatorname{Ker}}

\newcommand{\xra}{\xrightarrow}

\renewcommand{\geq}{\geqslant}
\renewcommand{\leq}{\leqslant}
\renewcommand{\ker}{\Ker}

\newcommand{\Hom}{\operatorname{Hom}}

\def\K{\mathcal{K}}
\def\Der{\mathrm{Der}}

\newcommand{\dd}{\mathbf{d}}
\newcommand{\ddd}{\overline{\mathbf{d}}}
\newcommand{\dddd}{\mathfrak{D}}
\newcommand{\hh}{\mathbf{h}}

\numberwithin{equation}{lem}

\begin{document}

\bibliographystyle{amsplain}

\title[]{On the semifree resolutions of DG algebras\\ over the enveloping DG algebras}

\author{Saeed Nasseh}
\address{Department of Mathematical Sciences\\
Georgia Southern University\\
Statesboro, GA 30460, U.S.A.}
\email{snasseh@georgiasouthern.edu}

\author{Maiko Ono}
\address{Institute for the Advancement of Higher Education, Okayama University of Science, Ridaicho, Kitaku, Okayama 700-0005, Japan}
\email{ono@ous.ac.jp}

\author{Yuji Yoshino}
\address{Graduate School of Natural Science and Technology, Okayama University, Okayama 700-8530, Japan}
\email{yoshino@math.okayama-u.ac.jp}

\thanks{Y. Yoshino was supported by JSPS Kakenhi Grant 19K03448.}


\keywords{bar resolution, DG algebra, DG module, diagonal ideal, enveloping DG algebra, lifting, reduced bar resolution, semifree resolution, tensor algebra, universal derivation.}
\subjclass[2010]{16E05, 16E45, 57T30.}

\begin{abstract}
The goal of this paper is to construct a semifree resolution for a non-negatively graded strongly commutative DG algebra $B$ over the enveloping DG algebra $B\otimes_AB$, where $A\subseteq B$ is a DG subalgebra and $B$ is semifree over $A$. Our construction of such a semifree resolution uses the notions of reduced bar resolution and tensor algebra of the shift of the diagonal ideal.
\end{abstract}

\maketitle


\section{Introduction}\label{sec20200314a}

This paper was initially motivated by our in-progress work in~\cite{NOY2} on the lifting theory of differential graded (DG) modules which has been significantly developed in the recent years by the authors for the purpose of studying the vanishing of Ext over commutative noetherian local rings. The goal of this paper is to construct a specific semifree resolution, which we denote by $(\mathbb{B},\mathbb{D})$, of a non-negatively graded strongly commutative DG algebra $B$ over the enveloping DG algebra $B\otimes_AB$, where $A\subseteq B$ is a DG subalgebra and $B$ is semifree over $A$. Our methods to construct such a semifree resolution $(\mathbb{B},\mathbb{D})$, which may be of independent interest for other possible applications, require the use of the notions of reduced bar resolution and tensor algebra of the shift of the diagonal ideal (which we refer to as the ``diagonal tensor algebra'' in~\cite{NOY2}).

After establishing the terminology and conventions in Section~\ref{sec20200314b'}, in Section~\ref{main section} we introduce the construction of $(\mathbb{B},\mathbb{D})$ and prove Theorem~\ref{thm20221205a}, which is our main result in this paper.
We also discuss how the lifting theory of DG modules is related to the notion of classical bar resolution in the appendix; see Theorem~\ref{thm20221109a}.

\section{Terminology and conventions}\label{sec20200314b'}

We assume that the reader is familiar with the basic concepts of DG homological algebra. The purpose of this section is to only specify some of the terminology and establish the conventions that will be used in the subsequent sections. General references on DG algebras, DG modules, and their properties include~\cite{avramov:ifr,avramov:dgha, felix:rht, GL}. For the unexplained terminology and facts, the reader may also consult~\cite{NOY1}.

\begin{para}\label{para20200329a}
Throughout the paper, $R$ is a commutative ring. By a differential graded (DG, for short) $R$-algebra we mean a non-negatively graded strongly commutative DG $R$-algebra. In this paper, $(A,d^A)$, or simply $A$, is a DG $R$-algebra with the underlying graded structure $A  = \bigoplus  _{n \geq 0} A _n$ and the differential $d^A$. For a homogeneous element $a\in A$, we denote by $|a|$ the degree of $a$.

The DG modules considered in this paper are right DG modules, unless otherwise stated.
For a DG $A$-module $(M, \partial^M)$, or simply $M$, we often refer to the graded $A$-module $M=\bigoplus_{i\in \mathbb{Z}}M_i$ as the underlying graded structure of $M$.
For an integer $i$, the $i$-th shift of a DG $A$-module $M$ is denoted by $\shift^i M$.
When $i=1$, we simply write $\shift M$ instead of $\shift^1 M$.
\end{para}

\begin{defn}
A DG $A$-module $P$ is \emph{semifree} if there is a well-ordered subset $F\subseteq P$ which is a basis for the underlying graded $A$-module $P$ such that for every element $f\in F$ we have $\partial^P(f)\in \sum_{e<f}eA$; for more details see~\cite{AH},~\cite[A.2]{AINSW}, or~\cite{felix:rht}. The set $F$ is called a \emph{semifree basis of $P$.}
A \emph{semifree resolution} of a DG $A$-module $M$ is a quasiisomorphism $P\xra{\simeq}M$, where $P$ is a semifree DG $A$-module.


\end{defn}

\begin{para}\label{assumption}
Throughout the paper, $(B,d^B)$, or simply $B$, is another DG $R$-algebra and $\varphi\colon A\to B$ is a homomorphism of DG $R$-algebras.
We further assume that $B$ is semifree as a DG $A$-module.
In this case, $1\in A$ gives a basis element of $B$ as an underlying graded free $A$-module and thus, $A$ is a DG $R$-subalgebra of $B$.
Let $\overline{B} := B/A$ and note that $\overline{B}$ is a semifree DG $A$-module as well.
\end{para}


\begin{ex}
For $n\leq \infty$, let
\begin{enumerate}[\rm(i)]
\item
$B=A\langle X_i\mid 1\leq i\leq n\rangle$ be a free extension of $A$, where $A$ is a divided power DG $R$-algebra; or
\item
$B=A[X_i\mid 1\leq i\leq n]$ be a polynomial extension of the DG $R$-algebra $A$.
\end{enumerate}
Then $A$ and $B$ satisfy the assumptions of~\ref{assumption}. In Case (i),  $\overline{B}$  is isomorphic to the ideal of  $B$ generated by all monomials except $1$ that contain the divided powers, and in Case (ii), $\overline{B}$  is isomorphic to the ideal of  $B$ generated by all monomials except $1$.
\end{ex}

\begin{ex} Assume that $A=R$ is a regular local ring and $I$ is an ideal of $R$. By taking a suitable quotient of the Tate resolution of $R/I$ over $R$, we see that there is a DG $R$-algebra $B$ which is finitely generated and semifree over $A$; see, for instance,~\cite[Proposition 2.2.8]{avramov:ifr} or~\cite{Tate}.
Again, $A$ and $B$ satisfy the assumptions of~\ref{assumption} and here, $\overline{B}$ has a finite semifree basis over $A$.
\end{ex}

\begin{para}\label{para20201114e}
Let $B^e:=B \otimes_A B$ denote the enveloping DG $R$-algebra of $B$ over $A$. Note that $B$ is a DG $R$-subalgebra of $B^e$ via the natural injection $B \to B^e$ defined by $b\mapsto b \otimes_A 1$ and DG $B^e$-modules are precisely the DG $(B, B)$-bimodules.

Following~\cite{NOY1}, the kernel of the DG $R$-algebra homomorphism $\pi_B\colon B^e\to B$ defined by $\pi_B(b\otimes_A b')=bb'$ is called the diagonal ideal and is denoted by $J$.
In this case, there is a short exact sequence
\begin{equation}\label{basic sequence}
0 \to J \xra{\iota} B^e \xra{\pi_B}  B\to 0
\end{equation}
of DG $B^e$-modules, where $\iota$ is the natural injection.
Note that $J$ is a semifree left and right DG $B$-module, but it is not semifree as a DG $B^e$-module.
\end{para}

\begin{para}
For each integer $n \geq 0$, let $J^{\otimes _B n}:=J \otimes _ B \cdots \otimes _ B J$ be the $n$-fold tensor product of $J$ over $B$ with the convention that $J ^{\otimes_B 0}=B$. Since  $J$  is free as an underlying graded left $B$-module, applying the functor $- \otimes _ B J^{\otimes _B n}$ to~\eqref{basic sequence}, we obtain the following short exact sequence of DG $B^e$-modules:
$$
0 \to J ^{\otimes _B (n+1)}\xra{\iota_n}  B \otimes _A J^{\otimes _B n} \xra{\pi_n} J^{\otimes _B n} \to 0.
$$
\end{para}

\begin{defn}
The well-defined $A$-linear map $\delta\colon B\to J$ that is given by the formula $\delta(b)=1\otimes_A b-b\otimes_A 1$ for all $b\in B$ is called the \emph{universal derivation}.
\end{defn}

In the following subsection, we discuss the notion of the reduced bar resolution of the DG $B^e$-module $B$ via the diagonal ideal $J$.

\subsection*{Reduced bar resolution}
For each integer $n\geq 0$, let $\overline{B} ^{\otimes_A n} := \overline{B} \otimes_A  \cdots \otimes_A  \overline{B}$ be the $n$-fold tensor product of $\overline{B}$ over $A$ with the convention that $\overline{B} ^{\otimes_A 0}=A$. Note that the DG $B^e$-module $B \otimes_A \overline{B}^{\otimes_A n} \otimes_A B$ is semifree.
In fact,  if  $\{ x _i \}_{i\in I}$  is a semifree basis of the DG $A$-module $\overline{B}$, then
$\{ 1 \otimes_A x_{i_1} \otimes_A \cdots \otimes_A x_{i_n} \otimes_A 1 \}_{\{i_j\in I\mid 1\leq j\leq n\}}$ (with some suitable ordering)  is a semifree basis of the DG $B^e$-module $B \otimes_A \overline{B}^{\otimes_A n} \otimes_A B$.

Now, let $\ddd_0:=\pi_B$ and for each integer $n\geq 1$, let $\ddd _n$
be the composition
$$
B \otimes _A  J^{\otimes _B n} \xra{\pi _n} J^{\otimes _B n} \xra{\iota _{n-1}} B \otimes _A  J^{\otimes _B (n-1)}.
$$
Each $B \otimes _A  J^{\otimes _B n}$  is a semifree DG $B^e$-module and $\ddd _n$  is a DG $B^e$-module homomorphism.
By definition, there is a long exact sequence
\begin{equation}\label{reduced bar}
\cdots \to B \otimes _A  J^{\otimes _B n} \xra{\ddd_n} B \otimes _A  J^{\otimes _B (n-1)} \to \cdots \to B \otimes _A  J \xra{\ddd_1}
B^e \xra{\ddd_0} B \to 0
\end{equation}
of DG $B^e$-modules which is a graded free resolution of $B$ as an underlying graded $B^e$-module.
Note that~\eqref{reduced bar} is a split exact sequence as one-sided (right or left) DG $B$-modules.

\begin{prop}\label{remark1}
For all $n\geq 0$, there exists an isomorphism $J^{\otimes _B n} \cong \overline{B} ^{\otimes_A n}  \otimes _A B$ of underlying right $B$-modules. Hence, \eqref{reduced bar} is a graded free resolution of $B$ consisting of the underlying graded free $B^e$-modules
$B \otimes_A  \overline{B} ^{\otimes_A n}  \otimes _A B$ with $n\geq 0$.
\end{prop}

\begin{proof}
The statement is clear for $n=0$. For $n=1$, note that the exact sequence~\eqref{basic sequence} of DG $B^e$-modules is split as a sequence of underlying right $B$-modules.
Such splitting comes from the identification of  $B$  as a right  $B$-submodule $A \otimes _A B$  of  $B^e$. Hence, using the snake lemma in the diagram
$$
\xymatrix{
0 \ar[r] & J  \ar[r] & B^e \ar[r] & B \ar[r]  & 0 \\
 && A \otimes_A B \ar[u] \ar[r]^<<<<<{\cong} & B \ar@{=}[u]& }
$$
we have an isomorphism  $J \cong \overline{B} \otimes _A B$ of underlying right $B$-modules. The isomorphism $J^{\otimes _B n} \cong \overline{B} ^{\otimes_A n}  \otimes _A B$ for $n\geq 2$ follows readily from this.
\end{proof}


\begin{defn}\label{defn20221118a}
The exact sequence~\eqref{reduced bar} of DG $B^e$-modules is called the \emph{reduced bar resolution} of $B$ and we denote it by $(\overline{\mathbf{B}},\overline{\dd})$.
\end{defn}

\section{Constructing a semifree resolution of the DG $B^e$-module $B$}\label{main section}

By Propositions~\ref{prop20221206a} (from the appendix) and~\ref{remark1}, the classical and reduced bar resolutions are semifree resolutions of $B$ over $B^e$. The purpose of this section is to construct a new semifree resolution $(\mathbb{B}, \mathbb{D})$ of the DG $B^e$-module $B$ using the reduced bar resolution and the tensor algebra of $\shift J$.

\subsection*{The underlying $B^e$-module structure $\mathbb{B}$}
Let $T$ denote the \emph{tensor algebra of $\shift J$ over $B$}, that is,
$$T:=\bigoplus _{n \geq 0} (\shift J) ^{\otimes_B n}=B \oplus \shift J \oplus \left(\shift J \otimes _B \shift J\right)  \oplus \left(\shift J  \otimes _B \shift J \otimes _B \shift J \right) \oplus \cdots.$$
Note that $T$ has an algebra structure by taking the tensor product as multiplication. At the same time, $T$ is a DG $B^e$-module with the differential $\partial^T$ defined by the Leibniz rule on $(\shift J) ^{\otimes_B n}$ by the formula
$$
\partial ^T (c_1 \otimes_B \cdots \otimes_B c_n ) = \sum _{i=1}^n (-1) ^{|c_1| + \cdots + |c_{i-1}|} c_1 \otimes_B \cdots \otimes_B \partial ^{\shift J} (c_i) \otimes_B \cdots \otimes_B c_n
$$
where  $c_i \in \shift J$ for all $1 \leq i \leq n$.
Now, let $$\mathbb{B}:= B \otimes _A T$$ and note that
$\mathbb{B}=\bigoplus _{n\geq 0} \left(B \otimes _A (\shift J) ^{\otimes_B n}\right)$ with the convention that $(\shift J) ^{\otimes_B 0}=B$.
Note also that for all integers $n\geq 0$ there is an isomorphism $B \otimes _A (\shift J) ^{\otimes_B n} \cong B \otimes _A \overline{B} ^{\otimes_A n}  \otimes _A B$ of graded $B^e$-modules by Proposition~\ref{remark1}. Hence, $\mathbb{B}$ is free as an underlying graded $B^e$-module.

\subsection*{The differential structure $\mathbb{D}$}
We can initially equip $\mathbb{B}$ with two differentials: one differential is just
$\partial ^{\mathbb{B}} := \partial ^{B \otimes _A T}$. More precisely, for each integer $n\geq 0$ and $b \otimes_A  \tau \in  B \otimes _A (\shift J)^{\otimes_B n}$ we have the equality
$$
\partial ^{\mathbb{B}}( b \otimes_A \tau ) =
d^B(b) \otimes_A \tau + (-1)^{|b|} b \otimes_A \partial ^T ( \tau ).
$$
To define the other differential, note that for each integer $n\geq 0$ we have an isomorphism $\shift^n(B\otimes_A J^{\otimes_Bn})\xra{\psi_n} B\otimes_A(\shift J)^{\otimes_Bn}$ defined by the formula
$$
\psi_n\left(b\otimes_A\tau_1\otimes_B\cdots\otimes_B\tau_n\right)= (-1)^{n|b|+\sum_{i=1}^{n-1}(n-i)|\tau_i|}b\otimes_A\tau_1\otimes_B\cdots\otimes_B\tau_n
$$
where $|\tau_i|$ is the degree of $\tau_i$ in $J$ and $\psi_0=\id_{B^e}$.
Now, the other differential on $\mathbb{B}$, which we denote by $\dddd=\{\dddd_n\}_{n\geq 1}\colon \mathbb{B} \to \shift\mathbb{B}$, is induced by the maps $\shift^n \ddd _n$ via the isomorphisms $\psi_n$, that is, for all $n\geq 1$ we have the following commutative diagram:
$$
\xymatrix{
\shift^n(B\otimes_A J^{\otimes_Bn})\ar[rr]^{\psi_n}\ar[d]_{\shift^n \ddd _n}&& B\otimes_A(\shift J)^{\otimes_Bn}\ar[d]^{\dddd_n}\\
\shift^n(B\otimes_A J^{\otimes_B{(n-1)}})\ar[rr]^-{\shift \psi_{n-1}}&&\shift(B\otimes_A(\shift J)^{\otimes_B{(n-1)}}).
}
$$

\begin{disc}\label{disc20230211a}
It is straightforward to check that $(\mathbb{B}, \partial ^{\mathbb{B}})$  is a semifree DG $B^e$-module and that $\dddd\colon (\mathbb{B}, \partial ^{\mathbb{B}}) \to (\shift \mathbb{B}, \partial ^{\shift \mathbb{B}})$ is a DG $B^e$-module homomorphism. In particular, we have
\begin{equation}\label{eq20221202a}
\dddd\partial ^{\mathbb{B}}=-\partial ^{\mathbb{B}}\dddd.
\end{equation}
Note that $\cdots\xra{\dddd_3}\mathbb{B}_2\xra{\dddd_2}\mathbb{B}_1\xra{\dddd_1}B^e\xra{\pi_B}B\to 0$
is in fact the reduced bar resolution of the DG $B^e$-module $B$.
\end{disc}

Now, we define a new differential $\mathbb{D}\colon  \mathbb{B} \to \shift \mathbb{B}$ on $\mathbb{B}$ by the formula
$$\mathbb{D} := \partial^{\mathbb{B}} + \dddd.$$ More precisely, $\mathbb{D}_0 = \partial^{\mathbb{B}}$ and for each integer $n\geq 1$ and an element $b \otimes_A  \tau \in  B \otimes _A (\shift J)^{\otimes_B n}$ we have the equality
$$
\mathbb{D}_n ( b \otimes_A \tau ) =
d^B(b) \otimes_A \tau + (-1)^{|b|} b \otimes_A \partial ^T ( \tau ) + \dddd _n ( b \otimes_A \tau).
$$

\begin{disc}\label{para20221201a}
For an integer $n\geq 0$, let $B ^{\otimes_A n} := B \otimes_A  B \otimes_A  \cdots \otimes_A  B$
be the $n$-fold tensor product of $B$ over $A$ with the convention that $B^{\otimes_A 0}=A$. Note that each $B^{\otimes_A n}$ naturally has a DG $R$-algebra structure with the naturally defined differential  $\partial ^{B ^{\otimes_A n} }$ (and hence, $B ^{\otimes_A 2} = B^e$). Recall that  $J$ is a DG ideal of $B ^{e}$ that is free as a one-sided underlying graded $B$-module.
Thus, $J \otimes _B J \subseteq  B^{\otimes_A 2} \otimes _B B^{\otimes_A 2} \cong B ^{\otimes_A 3}$.
By induction, we see that $J ^{\otimes _B n}  \subseteq  B ^{\otimes_A (n+1)}$ for all integers $n \geq1$.
Also, by Proposition~\ref{remark1}, there is an isomorphism  $\kappa\colon B \otimes _A \overline{B} \to J$ of underlying left $B$-modules such that $\kappa(b_1 \otimes_A \overline{b_2})=b_1\delta(b_2)$ for all $b_1,b_2\in B$.
Similarly, by induction, for each integer $n\geq 0$, one can introduce an isomorphism
$\kappa_n\colon B \otimes _A \overline{B}^{\otimes_A n}  \to J^{\otimes _B n}$ of underlying left $B$-modules such that $\kappa_1=\kappa$.

Let  $\{ b_{\lambda} \}_{\lambda\in \Lambda}\cup \{ 1\}$  be a semifree basis of the DG $A$-module $B$. Then, the set  $\{\overline{b_{\lambda}}\}_{\lambda\in \Lambda}$ is a basis of $\overline{B}$ as an underlying free $A$-module.
The isomorphism $\kappa$ shows that  $J$ is free as an underlying left $B$-module with the basis $\{ \delta (b_{\lambda}) \}_{\lambda\in \Lambda}$. Hence, $J^{\otimes _B n}$  is also free as an underlying left (resp. right) $B$-module with the basis
$$
\{ \delta (b_{\lambda _1}) \otimes_B \delta (b_{\lambda_2 }) \otimes _B \cdots \otimes_B  \delta (b_{\lambda _n}) \}_{\{\lambda_i\mid 1\leq i\leq n\}\subseteq \Lambda}.
$$
For each integer $n\geq 0$, note that the left $B$-module isomorphism $\kappa_n^{-1}$ is realized by mapping a basis element $\delta (b_{\lambda _1}) \otimes_B \delta (b_{\lambda_2 }) \otimes _B \cdots \otimes_B  \delta (b_{\lambda _n})$ to $1 \otimes_A \overline{b_{\lambda_1}} \otimes_A \cdots \otimes_A \overline{b_{\lambda_n}}$.
\end{disc}

\begin{prop}\label{lem20221202a}
$(\mathbb{B}, \mathbb{D})$ is a semifree DG $B^e$-module.
\end{prop}

\begin{proof}
Since  $\partial^{\mathbb{B}}$ satisfies the Leibniz rule and $\dddd$  is $B^e$-linear,
$\mathbb{D}$ satisfies the Leibniz rule as well.
The equality~\eqref{eq20221202a}  implies that $\mathbb{D}^2=0$.
Therefore, $(\mathbb{B}, \mathbb{D})$ is a DG $B^e$-module.
Using the notation from Remark~\ref{para20221201a}, we know that $J$ is free as an underlying right $B$-module with the basis $\{ \delta (b_{\lambda}) \}_{\lambda\in \Lambda}$. Hence, for each integer $n\geq 0$, the underlying $B^e$-module $B \otimes _A J^{\otimes _B n}$  is free with the basis
$$\{ 1 \otimes _A \delta (b_{\lambda _1}) \otimes_B \delta (b_{\lambda_2 }) \otimes _B \cdots \otimes_B  \delta (b_{\lambda _n}) \}_{\{\lambda_i\mid 1\leq i\leq n\}\subseteq \Lambda}.$$
Thus, we obtain a free basis of the underlying $B^e$-module  $\mathbb{B}$ that is a semifree basis with a suitable ordering on the index set.
\end{proof}

\subsection*{The semifree resolution $\alpha$}
Consider the map
$\alpha\colon (\mathbb{B}, \mathbb{D}) \to (B, d^B)$  which is defined as follows:
$$
\alpha:=
\begin{cases}
\ddd_0& \text{on}\ B^e\\
0& \text{on}\ B \otimes _A (\shift J) ^{\otimes _B n}\ \text{for all}\ n\geq 1.
\end{cases}
$$

\begin{prop}\label{lem20221205a}
$\alpha$ is a DG $B^e$-module homomorphism.
\end{prop}

\begin{proof}
To prove the assertion, it suffices to show that $\alpha\mathbb{D}= d^B\alpha$.
Let $n\geq 0$ be an integer and $b \otimes_A  \tau \in  B \otimes _A (\shift J)^{\otimes_B n}$.

If $n >1$, then $\alpha ( \mathbb{D} ( b \otimes_A \tau ) ) = 0$ by definition of $\alpha$ because in this case we have $ \mathbb{D} ( b \otimes_A \tau ) \in (B \otimes _A (\shift J)^{\otimes_B (n-1)})\oplus (B \otimes _A (\shift J)^{\otimes_B n})$. On the other hand, we also have $d^B (\alpha (b \otimes_A \tau)) =0$, again by definition of $\alpha$.

If $n=1$,  then
$\alpha (\mathbb{D}( b \otimes_A \tau ))= \ddd_0 (\ddd _1\psi_1^{-1}(b\otimes_A \tau)) =0=d^B (\alpha ( b \otimes_A \tau))$.

Finally, if $n=0$, then since $\tau \in B$,  we have
$$
\alpha (\mathbb{D}( b \otimes_A \tau )) = d^B(b)\tau + (-1) ^{|b|} bd^B(\tau) = d^B( b\tau ) = d^B (\alpha (b \otimes_A \tau))
$$
as desired.
\end{proof}

We are now ready to prove the main result of this paper.

\begin{thm}\label{thm20221205a}
$\alpha$ is a semifree resolution of the DG $B^e$-module $(B,d^B)$.
\end{thm}

\begin{proof}
We can regard  $\mathbb{B}$  as an underlying $\mathbb{Z} ^2$-graded right $R$-module. In fact, for  a homogeneous element  $\beta = b \otimes_A \tau \in  \mathbb{B}$ define
$$|\beta| =  (n, m) \in \mathbb{Z}^2\qquad \text{whenever}\qquad \tau \in \left((\shift J)^{\otimes_Bn}\right)_{m -|b|}.$$
Then,  $\partial ^{\mathbb{B}}$  is of degree $(0, -1)$ while  $\dddd$  is of degree $(-1, 0)$.
Note that each graded component $\mathbb{B}_{(n, m)}$ is an $R$-module.
Therefore, the complex  $(\mathbb{B}, \mathbb{D})$  is the total complex of the double complex $(\mathbb{B},  \partial ^{\mathbb{B}},  \dddd)$ of $R$-modules.
As we mentioned in Remark~\ref{disc20230211a},  the complex  $(\mathbb{B}, \dddd)$ is acyclic and hence, it is quasiisomorphic to $\HH_0(\mathbb{B}, \dddd) = B$.
Thus, the $E^2$ term of the first quadrant spectral sequence
$$
E_{p,q}^2:= \HH_p(\HH_q(\mathbb{B},  \dddd), \partial ^{\mathbb{B}})\overset{p}{\Longrightarrow} \HH_p(\mathbb{B},  \mathbb{D})
$$
degenerates and therefore, we have the isomorphisms
$$
\HH(\mathbb{B},  \mathbb{D}) \cong \HH(\HH_0(\mathbb{B},  \dddd), \partial ^{\mathbb{B}}) \cong \HH(B).
$$
This means that $\alpha$ is a quasiisomorphism and the assertion follows.
\end{proof}


\begin{disc}\label{prop20230114a}
We proved in Proposition~\ref{lem20221202a} that $(\mathbb{B}, \mathbb{D})$ is a semifree DG $B^e$-module. As we explain next, $(\mathbb{B}, \mathbb{D})$ has also a structure of a DG $T$-module.

Since  $\mathbb{B} = B \otimes _A T$, it naturally has a structure of underlying graded right $T$-module.
On the other hand, by definition, $\dddd$  is $T$-linear, i.e., for all $b \in B$ and $t_1, t_2 \in T$ we have
$\dddd ((b \otimes_A  t_1 )\otimes _B t_2) = \dddd (b \otimes _A t_1) \otimes _B t_2$.
Since  $\partial ^{\mathbb{B}}$ satisfies the Leibniz rule and $\dddd$ is $T$-linear, it is straightforward to see that  $\mathbb{D}$  satisfies the Leibniz rule as well, i.e.,
$\mathbb{D}((b \otimes _A t_1)\otimes _B t_2)= \mathbb{D}(b \otimes _A t_1) \otimes _B t_2+(-1)^{|b|+|t_1|}(b \otimes _A t_1) \otimes_B \partial ^T(t_2)$  for  $b \in B$ and $t_1, t_2 \in T$.

Note that, by definition, a DG $T$-module $M$ is just an underlying graded $B$-module $M$ with a DG $B$-module homomorphism $\mu\colon M \otimes_B \shift J \to M$,  which defines the action of $T$ on $M$ as
$m \cdot c = \mu (m \otimes_B c)$, for $m \in M$ and $c \in \shift J$.
Note that for the DG $B$-module $\mathbb{B} = B \otimes _A T$,
we have a natural DG $B$-module homomorphism
$$\mathbb{B} \otimes _B \shift J = B \otimes _A (T \otimes _B \shift J) \overset{\varepsilon}\hookrightarrow B \otimes _A T = \mathbb{B}.
$$
This inclusion map $\varepsilon$ defines the action of $T$ on $(\mathbb{B}, \mathbb{D})$, which is exactly the meaning of the previous paragraph regarding the DG $T$-module structure of $(\mathbb{B}, \mathbb{D})$.
\end{disc}

We conclude this section with the following discussion that is a consequence of Theorem~\ref{thm20221205a} and also is of a similar nature to our works in~\cite{NOY3, NOY2}.

\begin{disc}\label{disc20230102a}
For a semifree DG $B$-module $N$ with a semifree basis $\{ e_{\lambda}\}_{\lambda\in\Lambda}$, consider the right DG $B$-module homomorphism $\alpha_N:=\id_N \otimes_B \alpha$.
It follows from Theorem~\ref{thm20221205a} that $\alpha_N$ is a quasiisomorphism, i.e., $\alpha_N$ is an isomorphism in the homotopy category $\K(B)$.
Under the natural DG $B$-module isomorphisms $N \otimes _B (B, d^B)\cong (N,\partial^N)$ and $\varrho\colon N \otimes _B (\mathbb{B},\mathbb{D}) \xra{\cong}  (N\otimes _A T, \mathbb{D}_N)$, where $\mathbb{D}_N$ is the differential on $N\otimes _A T$ obtained naturally via $\varrho$, we can consider $\alpha_N$ as a  DG $B$-module homomorphism from the semifree DG $B$-module $(N\otimes _A T, \mathbb{D}_N)$ to the semifree DG $B$-module $N$.
Assume that
$\partial ^N (e_{\lambda} ) = \sum _{\mu < \lambda} e_{\mu} b_{\mu \lambda}$
is a matrix representation of the differential $\partial ^N$, where $b_{\mu \lambda}\in B$. We define a graded $B$-linear map $\beta_N\colon  N \to N \otimes _A T$ by the following formula:
\begin{align*}
\beta _N (e_{\lambda}) = & \
e_{\lambda} \otimes _A 1
+\sum _{\mu_1 < \lambda} e_{\mu _1} \otimes _A \delta (b_{\mu_1 \lambda} )
+\!\!\!\sum _{\mu_2 < \mu_1 < \lambda}\!\! e_{\mu _2} \otimes _A \delta (b_{\mu_2 \mu_1} ) \otimes_B \delta (b_{\mu_1 \lambda})\ +\\
&\cdots+\!\!\sum _{\mu_m < \cdots < \mu_1 < \lambda}\!\!\!\!\! e_{\mu _m} \otimes _A \delta (b_{\mu_m \mu_{m-1}} ) \otimes_B \cdots \otimes _B \delta (b_{\mu_2 \mu _1} ) \otimes _B  \delta (b_{\mu_1 \lambda} )\ +\cdots
\end{align*}
Since the index set $\Lambda$ is well-ordered and for each $\lambda\in\Lambda$ there is no infinite sequence $\cdots < \mu_2<\mu_1<\lambda$, the series of $\beta_N(e_{\lambda})$ terminates after finitely many terms.
Note that (surprisingly enough) the map $\beta_N$ contains a similar representation to the
obstruction of na\"{\i}ve lifting that is introduced in~\cite{NOY3}; see Appendix A for the definition of na\"{\i}ve lifting.

One can check that $\beta_N\colon  N \to N \otimes _A T$ is a DG $B$-module homomorphism that is the inverse morphism of $\alpha _N$ in $\K(B)$.
In fact, to prove that $\beta _N$ is a DG $B$-module homomorphism, we need to show that
$\beta _N \partial ^N = \mathbb{D}_N \beta _N$. This equality can be seen by evaluating both sides at basis elements $e_{\lambda}$, that is, one can check that $\beta _N(\partial ^N (e_{\lambda})) = \mathbb{D}_N(\beta _N(e_{\lambda}))$.
On the other hand, it is straightforward to see that $\alpha _N \beta _N = \id_N$.
Since $\alpha_N$ is an isomorphism in $\K(B)$, we have $\alpha _N ^{-1} = \beta _N$.
\end{disc}


\appendix
\section{Bar resolution and na\"{\i}ve lifting property}\label{appendixB}

Bar resolutions are well-known and extraordinarily useful constructions with a history that goes back quite far. These concepts are present in a wide range of algebraic and topological contexts; see for instance the works of Eilenberg and Mac Lane~\cite{EML, EML1} and Cartan~\cite{cartan, cartan1}.
In Section~\ref{main section} we used the notion of \emph{reduced} bar resolution to construct the semifree resolution $(\mathbb{B},\mathbb{D})$. In this appendix, we give a brief introduction to the notion of \emph{classical} bar resolution
and since our work in this paper is initially motivated by our study of the lifting theory of DG modules (see~\cite{NOY, NOY1, NOY3, nasseh:lql, nassehyoshino, OY} for the history and background), we discuss an application of the classical bar resolution in this theory; see Theorem~\ref{thm20221109a}.

\subsection*{Bar resolution}
We use the notation from Remark~\ref{para20221201a}. In this section, we regard $B ^{\otimes_A n}$ as a right DG $B^e$-module defined by the right action
$$
(b_1 \otimes_A b_2 \otimes_A  \cdots \otimes_A b_n) (b\otimes_A b') = (-1)^{|b|\left(|b_2|+\cdots+|b_n|\right)} b_1b  \otimes_A b_2 \otimes_A  \cdots \otimes_A b_nb'
$$
for all $b_1 \otimes_A b_2 \otimes_A  \cdots \otimes_A b_n\in B^{\otimes_A n}$ and $b\otimes_A b'\in B^e$.
With this description, for $n\geq 2$, we sometimes write $B \otimes_A B^{\otimes_A (n-2)} \otimes_A B$ instead of $B^{\otimes_A n}$ to emphasize that $B^{\otimes_A n}$ as an underlying graded $B^e$-module is generated by elements of the form $1 \otimes_A b_2 \otimes_A  \cdots \otimes_A b_{n-1}\otimes_A 1$.

Now, let $\dd _{-1}:=\pi_B$, where $\pi_B$ is the map defined in~\ref{para20201114e}. For integers $n\geq 1$, we define
$\dd _{n-1}\colon B \otimes_A B ^{\otimes_A n} \otimes_A B \to B \otimes_A B ^{\otimes_A (n-1)} \otimes_A B$
by the formula
\begin{align*}
\dd _{n-1} ( b_0 \otimes_A b_1 \otimes_A  \cdots \otimes_A b_{n+1})&=\\
\sum _{i =0}^{n} (-1)^{i}\ b_0 &\otimes_A \cdots \otimes_A b_{i-1} \otimes_A b_ib_{i+1} \otimes_A  b_{i+2}\otimes_A \cdots \otimes_A b_{n+1}.
\end{align*}
One can verify that for $n\geq 0$, the map $\dd _{n-1}$ is a DG $B^e$-module homomorphism such that $\dd _{n-1} \dd _{n} = 0$. Hence,
$$
(\mathbf{B},\dd):= \left(\cdots \xra{\dd _{n+1}} B \otimes_A B ^{\otimes_A (n+1)}  \otimes_A B \xra{\dd _{n}}  \cdots \xra{\dd _0} B ^{\otimes_A 2} \xra{\dd_{-1}} B \to 0\right)
$$
is a complex of DG $B^e$-modules.
\vspace{3pt}

The next result is very likely well-known; for instance, compare it to~\cite{Sri}.

\begin{prop}\label{prop20221206a}
The complex $(\mathbf{B},\dd)$ of DG $B^e$-modules is
\begin{enumerate}[\rm(a)]
\item a split exact sequence of right (and also left) DG $B$-modules; and
\item a (not necessarily split) exact sequence of DG $B^e$-modules.
\end{enumerate}
\end{prop}

\begin{proof}
For each integer $n\geq -1$, we define the map $\hh _n\colon  B ^{\otimes_A (n+2)}  \to  B ^{\otimes_A (n+3)}$ by
$$
\hh _n ( b_0 \otimes_A b_1 \otimes_A  \cdots \otimes_A b_{n+1}) = 1 \otimes_A  b_0 \otimes_A b_1 \otimes_A  \cdots \otimes_A b_{n+1}.
$$
Note that $\hh_n$ is a right DG $B$-module homomorphism and for all $n \geq 0$, one can check that the equalities
$\dd _{n} \hh _n + \hh_{n-1} \dd _{n-1} = \id _{B ^{\otimes_A (n+2)}}$ and $\dd _{-1} \hh _{-1} = \id _{B}$ hold.
This shows that the complex $(\mathbf{B},\dd)$ is a split exact sequence of right DG $B$-modules, and hence, $(\mathbf{B},\dd)$ is exact as a complex of DG $B^e$-modules.
Note that $\hh_n$ is only a right DG $B$-module homomorphism and hence, it is not a DG $B^e$-module homomorphism. Therefore, $(\mathbf{B},\dd)$ is not necessarily split as a complex of DG $B^e$-modules.

By symmetry, $(\mathbf{B},\dd)$ is a split exact sequence of left DG $B$-modules as well.
\end{proof}

\begin{defn}\label{defn20221108a}
The exact complex $(\mathbf{B},\dd)$ of DG $B^e$-modules is called the \emph{bar resolution} of the DG $B^e$-module $B$.
\end{defn}




\begin{para}\label{para20221115b'}
Let $\nu\colon B \to B^{\otimes_A 3}$ be a map defined by $\nu (b) = -(1 \otimes_A b \otimes_A 1)$ and note that $\delta = \dd _0 \nu$.
Let ${}^{0}\!J := B$ and for a positive integer $n$, let  ${}^n\!J := \Ker \dd _{n-2}$, which is a DG $B^e$-submodule of $B \otimes_A B^{\otimes_A (n-1)} \otimes_A B$.
Notice that ${}^1\!J \subseteq  B^{\otimes_A 2}$ is the diagonal ideal $J$ of $B$ over $A$ and it follows from the exactness of $(\mathbf{B},\dd)$ that ${}^n\!J = \im \dd_{n-1}$.
Therefore, for each $n \geq 1$, there is a short exact sequence
\begin{equation}\label{sex}
0 \to {}^n\!J \to B \otimes_A B ^{\otimes_A (n-1)}  \otimes_A B \to {}^{n-1}\!J \to 0
\end{equation}
of DG $B^e$-modules that is split as a sequence of right (resp. left) DG $B$-modules.
\end{para}

\begin{defn}\label{para20221115b''}
Let $L$ be a graded $B^e$-module.
A map $D\colon B \to  L$  is called an \emph{$A$-derivation} if it is $A$-linear and satisfies the equality $D (bb') = D(b)b' + bD(b')$  for all $b,  b' \in B$.
We denote by  $\Der _A (B, L)$ the set of all $A$-derivations from $B$ to $L$.
\end{defn}

\begin{prop}\label{para20221115b}
Let $L$ be a graded $B^e$-module. For a $B^e$-linear map $f\colon J \to L$, the map
$\eta\colon \Hom _{B^e} (J , L)  \to \Der _A (B, L)$
defined by $\eta(f)=f \delta$ is a bijection.
\end{prop}

This proposition which records the universal property of the diagonal ideal using the universal derivation and bar resolution
is classically well-known; see, for instance,~\cite[Proposition 2.4]{CQ}. However, we give a quick proof in this context for the reader's convenience.

\begin{proof}
For a $B^e$-linear map  $f\colon J \to L$, one can check that the composition $f\delta$ is an $A$-derivation, and hence, $\eta$ is well-defined.

Now, let $D\in \Der _A (B, L)$. Since $D$ is $A$-linear on both sides, we can define a $B^e$-linear map
$g := B \otimes _A D \otimes _A B : B^{\otimes_A 3} \to L$ by $g(b_0 \otimes_A b_1 \otimes_A b_2)=b_0 D(b_1)  b_2$.
The derivation property of $D$ implies
$g (\dd _1 ( b_0 \otimes_A b_1 \otimes_A b_2 \otimes_A b_3)) =0$.
Hence,  $\im \dd _1= {}^2\!J\subseteq \ker(g)$. It follows from~\eqref{sex} that $g$  induces a $B^e$-linear map $\overline{g}\colon J \to L$.
One can check that the correspondence $D \mapsto -\overline{g}$ is the inverse of $\eta$.
\end{proof}

\subsection*{Na\"{\i}ve liftability of DG modules} In this subsection, we discuss the relation between the lifting theory of DG modules and the classical bar resolutions.

\begin{defn}\label{para20221108e}
Let $N$ be a semifree DG $B$-module, and let $N |_A$ denote $N$ regarded as a DG $A$-module via the DG $R$-algebra homomorphism $\varphi$.
We say that $N$ is {\it na\"ively liftable  to $A$} if
the DG $B$-module epimorphism  $\pi_N\colon N |_A \otimes _A B \to N$
defined by $\pi_N(n \otimes b)=nb$ is split.
\end{defn}

\begin{disc}\label{para20221109a}
Let  $N$  be a semifree DG $B$-module. Assuming $B^{\otimes_A (-1)} \otimes_A B = A$, for all integers $n \geq -1$, there are natural isomorphisms
$N \otimes _B (B \otimes_A B ^{\otimes_A n}  \otimes_A B) \cong N \otimes _A B ^{\otimes_A n}  \otimes_A B$
of DG $B$-modules.
Hence, $N\otimes_B (\mathbf{B}, \dd)$ is of the form
$$
\cdots \xra{\dd ^N_{n}} N \otimes_A B ^{\otimes_A n}  \otimes_A B \xra{\dd ^N_{n-1}} \cdots \xra{\dd ^N_{1}} N \otimes _A B ^{\otimes_A 2 } \xra{\dd ^N _0}  N \otimes _A B \xra{\dd^N_{-1}} N \to 0
$$
where  $\dd^N_{n-1}$ is defined by the formula
\begin{align*}
\dd^N_{n-1} (x \otimes_A b_1 \otimes_A \cdots \otimes_A b_{n+1})
&=x b_1 \otimes_A \cdots \otimes_A b_{n+1}\\
& + \sum  _{i=1}^{n} (-1)^i x \otimes_A b_1 \otimes_A  \cdots \otimes_A b_ib_{i+1}\otimes_A \cdots \otimes_A b_{n+1}
\end{align*}
for $x \in N$  and  $b _i \in B$.
Since $(\mathbf{B},\dd)$ is an exact complex of DG $B$-modules and $N$ is free as an underlying graded $B$-module, the complex $N\otimes_B (\mathbf{B}, \dd)$ of DG $B$-modules is also exact.
Note that $\dd_{-1}^N=\pi_N$, where $\pi_N$ is the map defined in Definition~\ref{para20221108e}.
Hence, by definition, $N$ is na\"{\i}vely liftable to $A$  if and only if  $\dd_{-1}^N$ is a splitting DG $B$-module epimorphism.

Again, since $N$ is free as an underlying graded $B$-module, by virtue of~\eqref{sex}, for  all integers $n \geq 1$, we have the following short exact sequences of DG $B$-modules:
\begin{equation}\label{Nsex}
0 \to N \otimes _B {}^n\!J \to N \otimes _A B ^{\otimes_A (n-1) } \otimes_A B \to N \otimes _B {}^{n-1}\!J \to 0.
\end{equation}
\end{disc}

The main result of this section is the following.

\begin{thm}\label{thm20221109a}
The following are equivalent for a semifree DG $B$-module $N$:
\begin{enumerate}[\rm(i)]
\item
$N$ is na\"{\i}vely liftable to $A$;
\item
$\dd_{-1}^N$ is a splitting DG $B$-module epimorphism;
\item
The complex $N\otimes_B (\mathbf{B}, \dd)$ of DG $B$-modules is split exact, that is, for all integers $n \geq 1$, the short exact sequences~\eqref{Nsex} of DG $B$-modules split.
\end{enumerate}
\end{thm}

The proof of this theorem is given after the following lemma which can be verified easily using the fact that $\dd _n = 0$ for all $n \leq -2$. For this lemma, for positive integers $m, n$, note that the DG $B$-modules $(N \otimes _A B^{\otimes_A n} ) \otimes _B B ^{\otimes_A m}$ and $N \otimes _A B ^{\otimes_A (n+m-1)}$ are identified using the isomorphism $B^{\otimes_A n} \otimes _B B ^{\otimes_A m}\cong B ^{\otimes_A (n+m-1)}$.
In fact, considering the elements  $\xi = b_1 \otimes_A \cdots \otimes_A b_n \in B ^{\otimes_A n}$ and
$\xi' = b'_1 \otimes_A \cdots \otimes_A b'_m \in B ^{\otimes_A m}$, we have $\xi \otimes _B \xi' =b_1 \otimes_A \cdots \otimes_A b_n b'_1 \otimes_A \cdots \otimes_A b'_m$ as an element of $B ^{\otimes_A (n+m-1)}$.

\begin{lem}\label{lem20221112a}
Let $N$  be a semifree DG $B$-module and $m, n$ be positive integers. Then the following equalities hold for $\beta \in  B ^{\otimes_A n}$, $\beta ' \in  B ^{\otimes_A m}$, and $\gamma \in  N \otimes _A B ^{\otimes_A n}$:
\begin{align*}
\dd _{n+m-4} (\beta \otimes _B \beta' ) &= \dd _{n-3}(\beta) \otimes _B \beta' +(-1)^{n+1} \beta \otimes _B \dd_{m-3} (\beta')\\
\dd ^N_{n+m-3} (\gamma \otimes _B \beta' )& = \dd ^N_{n-2}(\gamma) \otimes _B \beta' +(-1)^{n} \gamma \otimes _B \dd_{m-3} (\beta').
\end{align*}
\end{lem}

\vspace{5pt}

\noindent \emph{Proof of Theorem~\ref{thm20221109a}.} The equivalence (i) $\Longleftrightarrow$ (ii) is from Remark~\ref{para20221109a}. Also, the implication (iii)$\implies$(ii) is trivial.

(ii)$\implies$(iii): By our assumption, there exists a DG $B$-module homomorphism $\rho\colon N \to N \otimes _A B$ such that
$\dd ^N_{-1} \rho = \id _N$.
For each integer $n \geq 2$, define the DG $B$-module homomorphism
$\lambda_n\colon N \otimes _B B ^{\otimes_A n} \to N \otimes _B B ^{\otimes_A (n+1)}$ by the formula
$$\lambda_n (x\otimes_B b_1 \otimes_A \cdots \otimes_A b_{n})
= \rho (x)\otimes_B b_1 \otimes_A b_{2} \otimes_A \cdots \otimes_A b_{n}.
$$
It follows from Lemma~\ref{lem20221112a} that
$\dd ^N_{n-2} (\lambda _n (x \otimes _B  \beta ) )= x \otimes _B \beta - \rho(x) \otimes _B  \dd _{n-3}(\beta)$
for all $x \in N$ and $\beta \in B ^{\otimes_A n}$.
Therefore, if we assume that $\beta \in {}^{n-1}\!J$, then
$$
\dd^N_{n-2}(\lambda _n (x \otimes _B \beta))  = x \otimes _B  \beta.
$$
This means that the restriction of the DG $B$-module homomorphism $\dd ^N_{n-2} \lambda _n$ on $N \otimes _B {}^{n-1}\!J$ is the identity map.
Thus, the sequence~\eqref{Nsex} is split.
\qed




\section*{Acknowledgments}
We are grateful to Ben Briggs, Srikanth Iyengar, and Josh Pollitz for useful
comments on this work regarding the history of bar resolutions and for introducing
some of the references.

\providecommand{\bysame}{\leavevmode\hbox to3em{\hrulefill}\thinspace}
\providecommand{\MR}{\relax\ifhmode\unskip\space\fi MR }
\providecommand{\MRhref}[2]{%
  \href{http://www.ams.org/mathscinet-getitem?mr=#1}{#2}
}
\providecommand{\href}[2]{#2}

\end{document}